\documentclass{article}
\usepackage{times}
\usepackage[hyperindex=true,pageanchor=true,hyperfigures=true,backref=false]{hyperref} 
\usepackage{amsmath}
\usepackage{amssymb}

\usepackage{url}
\usepackage{dsfont} 
\DeclareSymbolFontAlphabet{\Bbb}{AMSb}
\newlength{\myleftmargin}
\usepackage{amsmath}
\usepackage{amssymb}
\usepackage{amsthm}
\usepackage{color}
\DeclareSymbolFontAlphabet{\Bbb}{AMSb}

\newtheorem{theorem}{Theorem}

\newlength{\fixboxwidth}
\definecolor{darkgreen}{rgb}{0,0.6,0}

\newcommand{\N}{\mathbb{N}}

\newcommand{\R}{\mathbb{R}}

\DeclareMathOperator{\id}{id}

\newcommand{\snorm}[1]{\Vert #1 \Vert}

\newcommand{\inorm}[1]{\Vert #1 \Vert_\infty}

\newcommand{\ns}{m}
\newcommand{\xs}{X_*}
\newcommand{\xns}{X_{\ns}}

\newcommand{\cuc}{C_{\mathrm uc}(\xns)}
\newcommand{\cxs}{C(\xs)}
\newcommand{\hns}{H_m}
\newcommand{\kns}{k_m}

\title{Reproducing Kernel Hilbert Spaces Cannot Contain  all Continuous Functions on a Compact Metric Space}

\author{Ingo Steinwart\\
Institute for Stochastics and Applications\\
Faculty 8: Mathematics and Physics\\
University of Stuttgart\\
D-70569 Stuttgart Germany \\
\texttt{\small ingo.steinwart@mathematik.uni-stuttgart.de}
}

\begin{document}

\maketitle

\begin{abstract}
Given an uncountable, compact metric space, we show that there exists no reproducing
kernel Hilbert space that contains the space of all continuous functions on this compact space.
\end{abstract}

\vspace*{2ex}

Given a compact metric space $(X,d)$, we denote
the space of all continuous function $f:X\to \R$   by  $C(X)$.
If $k:X\times X\to \R$ is a continuous reproducing kernel, see e.g.~\cite{Aronszajn50a}
and \cite[Ch.~4]{StCh08} for an introduction, then it is well-known that its reproducing kernel Hilbert space (RKHS) $H$
satisfies $H\subset C(X)$. This raises the question, whether an inverse inclusion is also possible for  
suitable RKHSs. The following theorem provides a negative answer to this question.

\begin{theorem}
 Let $(X,d)$ be a compact metric space. If $X$ is uncountable, then there exists no 
 RKHS  $H$ on $X$ such that $C(X) \subset H$.
\end{theorem}

\begin{proof}
 Let us assume that there was an RKHS  $H$ on $X$ such that $C(X) \subset H$. Denoting 
 the kernel of $H$ by $k$, we define 
 \begin{align*}
  X_n := \bigl\{x\in X: k(x,x)\leq n\bigr\}
 \end{align*}
 for all $n\in \N$. Since $X= \bigcup_{n\geq 1} X_n$, we then find an $\ns\in \N$ such that 
 $\xns$ is uncountable. We write $\xs := \overline\xns$ for
  the closure of $\xns$ in $X$. Moreover, we define
  \begin{align*}
   \cuc := \bigl\{ f:\xns\to \R\,\,\bigl| \,\,f \mbox{ is uniformly continuous }\bigr\}\, .
  \end{align*}
 Then it is well known, see e.g.~\cite[p.~195]{Kelley55} that for every $f\in \cuc$ there is a unique, uniformly 
 continuous $\hat f:\xs\to \R$
 such that the restriction $\hat f_{|\xns}$ of $\hat f$ to $\xns$ satisfies 
 $\hat f_{|\xns} = f$. Since $\xs$ is compact, $\hat f$ is bounded, and hence $f$ is bounded. 
 Consequently, we can equip $\cuc$ with the supremums-norm $\inorm\cdot$, and we then easily
   conclude that this extension operation gives an isometric linear operator 
 \begin{align*}
  \hat \cdot\, : \cuc &\to \cxs \\
  f&\mapsto \hat f\, .
 \end{align*}
Moreover, the compactness of $\xs$ ensures that every $f\in \cxs$ is uniformly continuous, and hence this 
extension operator is actually an isometric isomorphism. 
In particular, $\cuc$ is a Banach space. In addition, \cite[Example 1.11.25]{Megginson98} shows that 
$\cxs$ is not reflexive since 
$\xs$ is infinite, and 
consequently, $\cuc$ is not reflexive.

% 
% In particular, $\cuc$ is a Banach space, and since
% $\cxs$ is separable,  see e.g.~\cite[Corollary 11.2.5]{Dudley02}, the space $\cuc$ is separable.
% For later use we further note that since $\xs$ is uncountable, the set of Dirac measures $\{\d_x: x\in \xs\}$
% is uncountable and satisfies $\snorm{\d_x - \d_y}_{\mathrm TV} \geq 1$ for all $x\neq y\in \xs$, where 
% $\snorm \cdot_{\mathrm TV}$ denotes the total variation norm on the space $\mathcal M(\xs)$ of finite, signed Radon  measures.
% Consequently, 

% 
% 
% By   a double application
% of \cite[III.D.19]{Wojtaszczyk91} together with $\dim C([0,1])=\infty$, 
% % or a direct construction using tent functions
% % centered on finer and finer packings of $\xs$, 
% we further observe that the fact that $\xs$ is uncountable implies
% $\dim \cxs = \infty$.

Let us now recall from e.g.~\cite[p.~351]{Aronszajn50a} that the RKHS of the restricted kernel
$\kns:=k_{|\xns\times \xns}$ is given by 
\begin{align*}
\hns := \bigl\{ h_{|\xns}\,\, \bigl|\,\, h\in H\bigr\}\, .
\end{align*}
Since $\kns$ is bounded by construction, namely $k(x,x) \leq \ns$ for all $x\in \xns$, we conclude 
that $\inorm h \leq \sqrt \ns \,\snorm h_{\hns}$ for all $h\in \hns$, see  
e.g.~\cite[Lemma 4.23]{StCh08}.

Our next goal is to show $\cuc \subset \hns$. To this end, we fix an $f\in \cuc$. Since $\xs$
is, as a metric space,  a normal space, see e.g.~\cite[Thm.~2.6.1]{Dudley02}, 
Tietze's extension theorem, see e.g.~\cite[Thm.~2.6.4]{Dudley02},
 then 
gives a  $g\in C(X)$ with $g_{|\xs} = \hat f$.
By our initial assumption $C(X)\subset H$ we then have $g\in H$, and hence we find the desired
$f =  \hat f_{|\xns} =     g_{|\xns} \in \hns$. 

Now recall that we have already seen that $\cuc$ is a Banach space.
A simple application of the closed graph theorem then shows that 
the inclusion map $\id: \cuc \to \hns$ is continuous. For the constant $K:= \snorm{\id: \cuc \to \hns} <\infty$ we thus have 
$\snorm f_{\hns} \leq K \inorm f$ for all $f\in \cuc$.

In summary, we have show that 
% By combining all these estimates we now conclude that 
for all $f\in \cuc$ we have 
 \begin{align*}
 \ns^{-1/2}\,   \inorm {f}  \leq  \snorm {f}_{\hns} \leq K \inorm f \, , 
 \end{align*}
i.e.~$ \snorm {\cdot}_{\hns}$ is an equivalent, and thus complete,  Hilbert space norm on $\cuc$. 
Consequently, $\cuc$ is reflexive, which contradicts our previous finding that $\cuc$ is not reflexive.
%
%
% By combining all these estimates we now conclude that for all $f\in \cxs$ we have 
%  \begin{align*}
%  \ns^{-1/2}\,   \inorm {f}  \leq  \snorm {f_{|\xns}}_{\hns} \leq K \inorm f \, .
%  \end{align*}
% Finally, to show that these inequalities lead to a contradiction, we fix a Rademacher sequence, i.e.~a sequence of i.i.d.~$\R$-valued random  variables 
% $(\e_i)_{i\geq 1}$ on some probability space $(\Omega,  \mathcal A, P)$ 
% such that $P(\{\e_i = -1\}) = P(\{\e_i = 1\}) = 1/2$ for all $i\geq 1$. 
% For $n\geq 1$ and arbitrary $f_1,\dots,f_n\in \cxs$ we then find by \cite[Cor.~11.8]{DiJaTo95} that 
% \begin{align*}
%  \Bigl( \E_P \bnorm {\sum_{i=1}^n \e_i f_i}_\infty^2 \,  \Bigr)^{1/2} 
%  &\leq \sqrt m \cdot  \Bigl( \E_P \bnorm {\sum_{i=1}^n \e_i (f_i)_{|\xns}}_{\hns}^2\,  \Bigr)^{1/2} \\
%  &\leq \sqrt m \cdot  \Bigl( \sum_{i=1}^n\mnorm{(f_i)_{|\xns}}_{\hns}^2\,      \Bigr)^{1/2}\\
%  &\leq \sqrt m \cdot K\cdot  \Bigl( \sum_{i=1}^n\mnorm{f_i}_{\infty}^2\,      \Bigr)^{1/2} \, .
% \end{align*}
% Consequently, $\cxs$ is a  type 2 space, but this contradicts \cite[Cor.~11.7]{DiJaTo95}
% in combination with \cite[Thm.~3.2]{DiJaTo95} and $\dim \cxs = \infty$.
\end{proof}

% \begin{corollary}
%  Let $(X,d)$ be a locally compact metric space 
% \end{corollary}
% 

Finally, we like to remark that 
the following theorem can be shown analogously.

\begin{theorem}
 Let $X$ be an uncountable set and $\ell_\infty(X)$ be the space of all bounded functions $f:X\to \R$. Then there exists no RKHS $H$ on $X$ such that  $\ell_\infty(X) \subset H$.
\end{theorem}

\noindent
\textbf{Acknowledgment.} I thank Daniel Winkle for bringing up this question and 
 Andrey Kharitenko for a valuable discussion that lead to a shorter proof.

\bibliographystyle{plain}
\bibliography{steinwart-mine,steinwart-books,steinwart-article}
% }

% \input{2dos}

% \input{checklist}

\end{document}